\documentclass[12pt,leqno]{amsart}
\usepackage{latexsym,amsmath,amssymb,mathrsfs,amsthm}
\usepackage{graphicx}

\title[The Dubovitski\u{\i}-Sard Theorem in Sobolev Spaces]{The Dubovitski\u{\i}-Sard Theorem in Sobolev Spaces}
\author{Piotr Haj\l{}asz and Scott Zimmerman}

\address{P.\ Haj{\l}asz: Department of Mathematics, University of Pittsburgh, 301
  Thackeray Hall, Pittsburgh, PA 15260, USA, {\tt hajlasz@pitt.edu}}

\address{S.\ Zimmerman: Department of Mathematics, University of Pittsburgh, 301
  Thackeray Hall, Pittsburgh, PA 15260, USA, {\tt srz5@pitt.edu}}

\thanks{P.H.\ was supported by NSF grant DMS-1161425.}

plus 6pt minus 12pt
\def\rank{{\rm rank\,}}

\def\eps{\varepsilon}

\def\id{{\rm id\, }}
\def\H{{\mathcal H}}

\newtheorem{theorem}{Theorem}
\newtheorem{lemma}[theorem]{Lemma}

\newtheorem{claim}[theorem]{Claim}

\def\diam{{\rm diam\,}}

\theoremstyle{definition}
\newtheorem{remark}[theorem]{Remark}

\newcommand{\barint}{
\rule[.036in]{.12in}{.009in}\kern-.16in \displaystyle\int }

\newcommand{\barcal}{\mbox{$ \rule[.036in]{.11in}{.007in}\kern-.128in\int $}}
\newcommand{\bbbn}{\mathbb N}
\newcommand{\bbbr}{\mathbb R}

\def\diam{\operatorname{diam}}

\def\mvint_#1{\mathchoice
          {\mathop{\vrule width 6pt height 3 pt depth -2.5pt
                  \kern -8pt \intop}\nolimits_{\kern -3pt #1}}%
          {\mathop{\vrule width 5pt height 3 pt depth -2.6pt
                  \kern -6pt \intop}\nolimits_{#1}}%
          {\mathop{\vrule width 5pt height 3 pt depth -2.6pt
                  \kern -6pt \intop}\nolimits_{#1}}%
          {\mathop{\vrule width 5pt height 3 pt depth -2.6pt
                  \kern -6pt \intop}\nolimits_{#1}}}

\numberwithin{theorem}{section} \numberwithin{equation}{section}

\begin{document}

\sloppy

%\dsp

\sloppy

%--------------------------------------------------------

\begin{abstract}

The Sard theorem from 1942 requires that a mapping $f:\mathbb{R}^n \to \mathbb{R}^m$ is of class $C^k$, $k > \max (n-m,0)$. 
In 1957 Duvovitski\u{\i} generalized Sard's theorem  to the case of $C^k$ mappings for all $k$. 
Namely he proved that, for almost all $y\in \mathbb{R}^m$, $\H^{\ell}(C_f \cap f^{-1}(y))=0$ where $\ell = \max(n-m-k+1,0)$, $\H^{\ell}$ 
denotes the Hausdorff measure, and $C_f$ is the set of critical points of $f$. 
In 2001 De Pascale proved that the Sard theorem holds true for Sobolev mappings of the class 
$W_{\rm loc}^{k,p}(\bbbr^n,\mathbb{R}^m)$, $k>\max(n-m,0)$ and $p>n$. We will show that also Dubovitski\u{\i}'s theorem 
can be generalized to the case of $W_{\rm loc}^{k,p}(\bbbr^n,\mathbb{R}^m)$ mappings for all $k\in\bbbn$ and $p>n$.

\end{abstract}

\maketitle

\section{Introduction}

Originally proven in 1942, Arthur Sard's \cite{sard} famous theorem asserts that the set of critical values of a sufficiently 
regular mapping is null. We will use the following notation to represent the \emph{critical set} of a given smooth map $f: \mathbb{R}^n \to \mathbb{R}^m$:
$$
C_f = \{x \in \mathbb{R}^n \; | \; \mathrm{rank} \, Df(x) < m \}.
$$
Throughout this paper, we will assume that $m$ and $n$ are integers at least 1.
Most notation used in the introduction is carefully explained in Section~\ref{notation}.

\begin{theorem}[Sard]
Suppose $f: \mathbb{R}^n \to \mathbb{R}^m$ is of class $C^k$. If $k > \mathrm{max}(n-m,0)$, then 
$$
\mathcal{H}^m (f(C_f)) = 0.
$$
\end{theorem}
Here and in what follows by $\H^k$ we denote the $k$-dimensional Hausdorff measure.

Several results have shown that Sard's result is optimal, see e.g. \cite{dubov,Grinberg,HajWhit,Kaufman,malysz,Whitney}.
In 1957 Dubovitski\u{\i} \cite{dubov},
extended Sard's theorem to all orders of smoothness $k$. See \cite{BojHaj} for a modernized proof of ths result
and some generalizations.

\begin{theorem}[Dubovitski\u{\i}]
\label{dubovClassical}
Fix $n,m,k \in \mathbb{N}$. Suppose $f: \mathbb{R}^n \to \mathbb{R}^m$ is of class $C^k$. Write $\ell = \max(n-m-k+1,0)$. Then 
$$
\mathcal{H}^{\ell}(C_f \cap f^{-1}(y)) = 0
\quad
\text{for a.e. $y \in \mathbb{R}^m$.}
$$
\end{theorem}

This result tells us that almost every level set of a smooth mapping intersects with its critical set on an $\ell$-null set. 
Higher regularity of the function implies a reduction in the Hausdorff dimension of the overlap between 
$f^{-1}(y)$ and $C_f$ for a.e. $y \in \mathbb{R}^m$.

Notice that if $k > \mathrm{max}(n-m,0)$, then $n-m-k+1 \leq 0$,
and so $\H^{\ell}=\H^{0}$ is simply the counting measure on 
$\bbbr^n$. That is, if $f: \mathbb{R}^n \to \mathbb{R}^m$ is of class $C^k$ and additionally $k > \mathrm{max}(n-m,0)$, 
Dubovitski\u{\i}'s theorem implies that $f^{-1}(y) \cap C_f$ is empty for almost every $y \in \mathbb{R}^m$. In other words,
$\H^m (f(C_f)) = 0$.
Thus Sard's theorem is a special case of Dubovitski\u{\i}'s theorem.

Recently, many mathematicians have worked to generalize Sard's result to the class of Sobolev mappings \cite{Alberti,BojHaj,Bourgain,Bourgain2,Depascale,Figalli,Korobkov,vander}.
Specifically, in 2001 De Pascale \cite{Depascale} proved the following version of Sard's theorem for Sobolev mappings.

\begin{theorem}
\label{dePascale}
Suppose $k > \max(n-m,0)$. Suppose $\Omega \subset \bbbr^n$ is open. If $f \in W_{\rm loc}^{k,p}(\Omega,\mathbb{R}^m)$ for $n<p<\infty$, then $\mathcal{H}^m(f(C_f))=0$.
\end{theorem}

In this paper we will use the usual notation $W^{k,p}(\mathbb{R}^n,\mathbb{R}^m)$ to indicate the Sobolev class of $L^p(\mathbb{R}^n,\mathbb{R}^m)$ 
mappings whose first $k$ weak partial derivatives have finite $L^p$ norm.

The purpose of this paper is to show that also the Dubovitski\u{\i} theorem generalizes to the case of $W_{\rm loc}^{k,p}$ mappings when $n<p<\infty$.
We must be very careful when dealing with Sobolev mappings because the set $f^{-1}(y)$ depends on what representative of $f$ we take.
If $k \geq 2$, then Morrey's inequality implies that $f$ has a representative of class $C^{k-1,1-\frac{n}{p}}$, so the critical set $C_f$ is well defined.
If $k=1$, then $Df$ is only defined almost everywhere and hence the set $C_f$ is defined up to a set of measure zero.
We will say that $f$ is \emph{precisely represented} if each component $f_i$ of $f$ satisfies
$$
f_i(x) = \lim_{r \to 0} \frac{1}{|B(x,r)|} \int_{B(x,r)} f_i(y) \, dy
$$
for all $x \in \Omega$ at which this limit exists. The Lebesgue differentiation theorem ensures that this is indeed a well defined representative of $f$.
In what follows, we will always refer to the $C^{k-1,1-\frac{n}{p}}$ representative of $f$ when $k \geq 2$ and a precise representation of $f$ when $k=1$. 
(Notice that the precise representative of $f$ and the smooth representative of $f$ are the same for $k \geq 2$.)

The main result of the paper reads as follows.
\begin{theorem}
\label{mainTheorem}
Fix $n,m,k \in \mathbb{N}$. Suppose $\Omega \subset \bbbr^n$ is open and $f\in W_{\rm loc}^{k,p}(\Omega,\mathbb{R}^m)$ for some $n<p<\infty$. If $\ell = \max(n-m-k+1,0)$, then
$$
\mathcal{H}^{\ell}(C_f \cap f^{-1}(y))=0
\quad
\text{for a.e. $y \in \mathbb{R}^m$.}
$$
\end{theorem}

If $m>n$, then since $p>n$ we may apply Morrey's inequality combined with H\"{o}lder's inequality to show that $\H^n(f(Q)) < \infty$ for any cube 
$Q \Subset \Omega$, and so $\H^m(f(\Omega)) = 0$. Thus $f^{-1}(y)$ is empty for almost every $y \in \bbbr^m$, and the theorem follows.

We will now discuss the details behind the argument that $\H^n(f(Q)) < \infty$ for any cube $Q \Subset \Omega$.
Fix $\delta>0$, and cover $Q$ with $2^{n\nu}$ congruent dyadic cubes $\{ Q_j \}_{j=1}^{2^{n\nu}}$ with pairwise disjoint interiors. 
According to Morrey's inequality (see Lemma~\ref{L2}),
$$
\diam f(Q_j) \leq C (\diam Q_j)^{1-\frac{n}{p}}\left(\int_{Q_j} |Df(z)|^p\, dz\right)^{1/p}
$$
for every $1\leq j\leq 2^{n\nu}$. 
Since $\diam Q_{j}=2^{-\nu}\diam Q$,
choosing $\nu$ large enough
gives $\sup_{j} \diam f(Q_j) < \delta$, and so we can estimate the pre-Hausdorff measure
\begin{align*}
\H_{\delta}^n(f(Q)) &\leq C \sum_{j=1}^{2^{n\nu}} (\diam f(Q_j))^n \\
&\leq 
C \sum_{j=1}^{2^{n\nu}} (\diam Q_j)^{n(1-\frac{n}{p})} \left(\int_{Q_j} |Df(z)|^p\, dz\right)^{n/p} \\
&\leq 
C \left( \sum_{j=1}^{2^{n\nu}} (\diam Q_j)^n \right)^{1-\frac{n}{p}} \left( \sum_{j=1}^{2^{n\nu}} \int_{Q_j} |Df(z)|^p\, dz \right)^{n/p} \\
&\leq 
C \H^n(Q)^{1-\frac{n}{p}} \left( \int_{Q} |Df(z)|^p\, dz \right)^{n/p}.
\end{align*}
We used H\"{o}lder's inequality with exponents $p/n$ and $p/(p-n)$ to obtain the third line. 
Since the right hand estimate does not depend on $\delta$, sending $\delta \to 0^+$ yields $\H^n(f(Q)) < \infty$. 
This completes the proof of Theorem~\ref{mainTheorem} when $m>n$. Hence we may assume that $m \leq n$.

We will now discuss the case $k=1$ to avoid any confusion involving the definition of $C_f$.
Since $m \leq n$, we may apply the following co-area formula due to Mal\'y, Swanson, and Ziemer \cite{malysz}:
\begin{theorem}
\label{coarea}
Suppose that $1 \leq m \leq n$, $\Omega \subset \bbbr^n$ is open, $p>m$, and $f \in W_{\rm loc}^{1,p}(\Omega,\bbbr^m)$ is precisely represented. 
Then the following holds for all measurable $E \subset \Omega$:
$$
\int_E |J_mf(x)| \, dx = \int_{\bbbr^m} \mathcal{H}^{n-m}(E \cap f^{-1}(y)) \, dy
$$
where $|J_mf|$ is the square root of the sum of the squares of the determinants of the $m \times m$ minors of $Df$.
\end{theorem}
Notice that $|J_mf|$ is equals zero almost everywhere on the set $E = C_f$.
Therefore the above equality with $E = C_f$ reads
$$
0 = \int_{\bbbr^m} \mathcal{H}^{n-m}(C_f \cap f^{-1}(y)) \, dy=\int_{\bbbr^m} \H^\ell(C_f\cap f^{-1}(y))\, dy.
$$
That is, $\mathcal{H}^{\ell}(C_f \cap f^{-1}(y)) = 0$ for a.e. $y \in \bbbr^m$, and the theorem follows.

{\em Therefore, we may assume for the remainder of the paper that $m \leq n$ and $k \geq 2$.}

Most proofs of Sard-type results typically involve some form of a Morse Theorem \cite{morse} in which the critical set of a mapping is decomposed into 
pieces on which the function's difference quotients converge quickly. See \cite{sternberg} for the proof of the classical Sard theorem based on this method.
A version of the Morse Theorem was also used by De Pascale \cite{Depascale}.
However, there is another approach to the Sard theorem based on the so called
Kneser-Glaeser Rough Composition theorem, and this method entirely avoids the use of the Morse theorem. 
We say that a mapping $f:W \subset \mathbb{R}^r \to \mathbb{R}$ of class $C^k$ is \emph{$s$-flat} on $A \subset W$ for 
$1 \leq s \leq k$ if $D^{\alpha}f=0$ on $A$ for every $1 \leq |\alpha| \leq s$. 

\begin{theorem} [Kneser-Glaeser Rough Composition] 
\label{rough}
Fix positive integers $s,k,r,n$ with $s<k$. Suppose $V \subset \bbbr^r$ and $W \subset \bbbr^n$ are open.
Let $g:V \to W$ be of class $C^{k-s}$ and $f:W \to \bbbr$ be of class $C^k$.
Suppose $A^* \subset V$ and $A \subset W$ are compact sets with
\begin{enumerate}
\item $g(A^*) \subset A$ and
\item $f$ is $s$-flat on $A$.
\end{enumerate}
Then there is a function $F:\bbbr^r \to \bbbr$ of class $C^k$ so that $F = f \circ g$ on $A^*$ and $F$ is $s$-flat on $A^*$.
\end{theorem}

This theorem ensures that the composition of two smooth maps will have the same regularity as the second function involved in the composition provided that enough of 
the derivatives of this second function are zero. After a brief examination of the rule for differentiation of composite functions, such a conclusion seems very natural.
Indeed, we can formally compute $D^{\alpha}(f \circ g)(x)$ for all $|\alpha| \leq k$ and $x \in A^*$ since any ``non-existing'' 
derivative $D^{\beta}g(x)$ with $|\beta| > k-s$ is multiplied by a vanishing 
$D^{\gamma} f(g(x))$ term with $|\gamma|=|\alpha|-|\beta| < s$. Thus we can formally set $D^{\gamma}f(g(x)) D^{\beta}g(x)=0$.
However the proof of this theorem is not easy since it is based on the celebrated Whitney extension theorem. That should not be surprising after all. 
The existence of the extension $F$ is proven by verification that the formal jet of derivatives of $f\circ g$ up to order $k$ defined above satisfies the assumptions of the Whitney extension theorem.
 
In 1951, Kneser presented a proof of this composition 
result in \cite{Kneser}. In the same paper, he proved a theorem which may be obtained as an immediate corollary to the theorem of Sard, though he did so without 
any reference to or influence from Sard's result. The composition theorem is also discussed in a different context in a 1958 paper by Glaeser \cite{Glaeser}. 
The reader may find the proofs of this theorem in \cite[Theorem~14.1]{Transversal}, \cite[Chapter~1, Theorem~6.1]{Malgrange}, \cite[Theorem~8.3.1]{DifferentialTopology}.

Thom \cite{Thom}, quickly realized that the method of Kneser can be used to prove the Sard theorem. See also 
\cite{Transversal,Malgrange,Metivier}.
Recently Figalli \cite{Figalli} used this method to provide a simpler proof of Theorem~\ref{dePascale}.
Our proof of Theorem~\ref{mainTheorem} we will also be based on the Kneser-Glaeser result.

\section{Notation and auxiliary results} \label{notation}

In this section we will explain notation and prove some technical results related to the Morrey inequality that will be used in the proof of Theorem~\ref{mainTheorem}.

Consider $f:\bbbr^n \to \bbbr$. By $D^\alpha f$ we will denote the partial derivative of $f$ with respect to the multiindex 
$\alpha=(\alpha_1,\ldots,\alpha_n)$. 
In particular $D^{\delta_i}f=\partial f/\partial x_i$, i.e. $\delta_i=(0,\ldots,0,1,0,\ldots,0)$ is a multiindex with $1$ on $i$th position.
Also $|\alpha|=\alpha_1+\ldots+\alpha_n$ and $\alpha!=\alpha_1!\cdots\alpha_n!$. $D^k f$ will denote the vector
whose components are the derivatives $D^\alpha f$, $|\alpha|=k$. 
The classes of functions with continuous and $\alpha$-H\"older continuous derivatives of order up to $k$ will be denoted by $C^k$ and $C^{k,\alpha}$ respectively.
The integral average over a set $S$ of positive measure will be denoted by
$$
f_S = \barint_S f(x)\, dx = \frac{1}{|S|}\int_S f(x)\, dx.
$$
The characteristic function of a set $E$ will be denoted by $\chi_E$.
The $k$-dimensional Hausdorff measure will be denoted by $\H^k$. 
In particular $\H^0$ is the counting measure.
The Lebesgue measure in $\bbbr^n$ coincides with $\H^n$.
In addition to the Hausdorff measure notation we will also write $|S|$ for the Lebesgue
measure of $S$. 
We say that a set is $k$-null if its $k$-dimensional Hausdorff measure equals zero.
By $\H^k_\delta$, $\delta>0$, we denote the pre-Hausdorff measure defined by taking infimum over coverings of the set
by sets of diameters less than $\delta$ so $\H^k(E)=\lim_{\delta\to 0^+} \H^k_\delta(E)$.
Cubes in $\bbbr^n$ will always have sides parallel to coordinate directions.
The symbol $C$ will be used to represent a generic constant and the actual value of $C$ may change in a single string of estimates. By writing $C=C(n,m)$ 
we indicate that the constant $C$ depends on $n$ and $m$ only.

We will use the following elementary result several times.
\begin{lemma}
\label{L6}
Let $E\subset\bbbr^n$ be a bounded measurable set and let $-\infty<a<n$. Then there is a constant $C=C(n,a)$
such that for every $x\in E$
$$
\int_E \frac{dy}{|x-y|^a} \leq
\begin{cases}
           C|E|^{1-\frac{a}{n}} \quad \text{if $0\leq a<n$.} \\
           (\diam E)^{-a}|E|   \quad \text{if $a<0$}.
\end{cases}
$$
\end{lemma}
\begin{proof}
The case $a<0$ is obvious since then $|x-y|^{-a}\leq (\diam E)^{-a}$.
Thus assume that $0\leq a < n$. In this case the inequality is actually true for all $x\in\bbbr^n$ and not only
for $x\in E$. Let $B=B(0,r)$, $|B|=|E|$. We have
$$
\int_E\frac{dy}{|x-y|^a} \leq \int_B \frac{dy}{|y|^a} = C\int_0^r t^{-a} t^{n-1}\, dt = 
C r^{n-a} = C |E|^{1-\frac{a}{n}}.
$$
\end{proof}

The following result \cite[Lemma~7.16]{GT} is a basic pointwise estimate for Sobolev functions. 
\begin{lemma}
\label{L1}
Let $D\subset\bbbr^n$ be a cube or a ball and let $S\subset D$ be a measurable set of positive measure. If $f\in W^{1,p}(D)$, $p\geq 1$, then
\begin{equation}
\label{e1}
|f(x)-f_S| \leq C(n)\, \frac{|D|}{|S|} \int_D \frac{|Df(z)|}{|x-z|^{n-1}}\, dz
\quad
\text{a.e.}
\end{equation}
\end{lemma}
When $p>n,$ the triangle inequality $|f(y)-f(x)|\leq |f(y)-f_D|+|f(x)-f_D|$, H\"older inequality, and Lemma~\ref{L6} applied to the right hand side of \eqref{e1} yield a well known
\begin{lemma}[Morrey's inequality]
\label{L2}
Suppose $n<p<\infty$ and $f\in W^{1,p}(D)$, where $D\subset\bbbr^n$ is a cube or a ball. Then
there is a constant $C=C(n,p)$ such that
$$
|f(y)-f(x)|\leq C (\diam D)^{1-\frac{n}{p}} \left(\int_D |Df(z)|^p\, dz\right)^{1/p}
\quad
\text{for all $x,y\in D$.}
$$
In particular,
$$
\diam f(D) \leq C (\diam D)^{1-\frac{n}{p}}\left(\int_D |Df(z)|^p\, dz\right)^{1/p}.
$$
\end{lemma}
Since $p>n$, the function $f$ is continuous (Sobolev embedding) and hence the lemma
does indeed hold for \emph{all} $x,y\in D$.

From this lemma we can easily deduce a corresponding result for higher order derivatives. The Taylor polynomial and the averaged Taylor polynomial of $f$ will be denoted by
$$
T_x^k f(y) = \sum_{|\alpha|\leq k} D^\alpha f(x)\, \frac{(y-x)^\alpha}{\alpha!},
\quad
T_S^k f(y) = \barint_S T_x^kf(y)\, dx.
$$
\begin{lemma}
\label{L3}
Suppose $n<p<\infty$, $k\geq 1$ and $f\in W^{k,p}(D)$, where $D\subset\bbbr^n$ is a cube or a ball. Then
there is a constant $C=C(n,k,p)$ such that
$$
|f(y)- T_x^{k-1} f(y)|\leq C(\diam D)^{k-\frac{n}{p}}\left( \int_D |D^k f(z)|^p\, dz\right)^{1/p}
\quad
\text{for all $x,y\in D$.}
$$
\end{lemma}
\begin{proof}
Given $y\in D$ let
$$
\psi(x) := T_x^{k-1}f(y) = \sum_{|\alpha|\leq k-1} D^\alpha f(x)\, \frac{(y-x)^\alpha}{\alpha!}\in W^{1,p}(D).
$$
Observe that $\psi(y)=f(y)$ and
$$
\frac{\partial\psi}{\partial x_j}(x) = \sum_{|\alpha|=k-1} D^{\alpha+\delta_j} f(x)\, \frac{(y-x)^\alpha}{\alpha!}\, ,
$$
where $\delta_j=(0,\ldots,1,\ldots,0)$. Indeed, after applying the Leibniz rule to $\partial\psi/\partial x_j$
the lower order terms will cancel out. Since
$$
|D\psi(z)|\leq C(n,k) |D^k f(z)||y-z|^{k-1},
$$
Lemma~\ref{L2} applied to $\psi$ yields the result.
\end{proof}
Applying the same argument to Lemma~\ref{L1} leads to the following result, see
\cite[Theorem~3.3]{BojHaj}.
\begin{lemma}
\label{L4}
Let $D\subset\bbbr^n$ be a cube or a ball and let $S\subset D$ be a measurable set of positive measure. If $f\in W^{k,p}(D)$, $p\geq 1$, $k\geq 1$, then
there is constant $C=C(n,k)$ such that
\begin{equation}
\label{e2}
|f(x)-T_S^{k-1} f(x)|\leq
C\, \frac{|D|}{|S|} \int_D \frac{|D^k f(z)|}{|x-z|^{n-k}}\, dz
\quad
\text{for a.e. $x\in D$.}
\end{equation}
\end{lemma}

In the next result we will improve the above estimates under the additional assumption that the derivative $Df$ vanishes on a given subset of $D$.
For a similar result in a different setting see \cite[Proposition~2.3]{HMz}.
\begin{lemma}
\label{L5}
Let $D\subset\bbbr^n$ be a cube or a ball and let $f\in W^{k,p}(D)$, $n<p<\infty$, $k\geq 1$. Let
$$
A=\{ x\in D|\, Df(x)=0\}.
$$
Then for any $\eps>0$ there is $\delta=\delta(n,k,p,\eps)>0$ such that if
$$
\frac{|D\setminus A|}{|D|} <\delta,
$$
then
$$
\diam f(D) \leq \eps (\diam D)^{k-\frac{n}{p}}\left(\int_D |D^k f(z)|^p\, dz\right)^{1/p}\, .
$$
\end{lemma}
\begin{remark}
\label{Rem1}
It is important that $\delta$ does not depend of $f$. The result applies very well to density points of $A$. 
Indeed, it follows immediately that if $x\in A$ is a density point, then for any $\eps>0$ there is $r_x>0$
such that
$$
\diam f(B(x,r_x)) \leq \eps r_x^{k-\frac{n}{p}}
\left(\int_{B(x,r_x)} |D^k f(z)|^p\, dz\right)^{1/p}\, .
$$
\end{remark}
\begin{proof}[Proof of Lemma~\ref{L5}]
Although only the first order derivatives of $f$ are equal zero in $A$, it easily follows that
$D^\alpha f=0$ a.e. in $A$ for all $1\leq |\alpha|\leq k$. Indeed, if a Sobolev function is constant in a set, its derivative equals zero
a.e. in the set, \cite[Lemma~7.7]{GT}, and we apply induction. Hence
$$
T_A^{k-1}f(x) = f_A
\quad
\text{for all $x\in\bbbr^n$.}
$$
Let $\eps > 0$. 
Choose $0<\delta<1/2$ with $\max \left\{ \delta^{\frac{k}{n}-\frac{1}{p}}, \delta^{1-\frac{1}{p}} \right\} < \varepsilon$. Since $\delta<1/2$, $|D|/|A|<2$. 
Thus Lemma~\ref{L4} with $S=A$ yields
$$
|f(x)-f_A|\leq 
C(n) \int_{D\setminus A} \frac{|D^k f(z)|}{|x-z|^{n-k}}\, dz \leq
C(n) \Vert D^k f\Vert_{L^p(D)}
\left( \int_{D\setminus A} \frac{dz}{|x-z|^{(n-k)\frac{p}{p-1}}}\right)^{\frac{p-1}{p}}\, .
$$
Now the result follows directly from Lemma~\ref{L6}. Indeed, if $k\leq n$, Lemma~\ref{L6} and the estimate
$$
|D\setminus A|< \delta |D|\leq C(n) \delta (\diam D)^n
$$
yield
$$
\left( \int_{D\setminus A} \frac{dz}{|x-z|^{(n-k)\frac{p}{p-1}}}\right)^{\frac{p-1}{p}} \leq
C(n,k,p) |D\setminus A|^{\frac{1}{n}(k-\frac{n}{p})} \leq
C(n,k,p) \delta^{\frac{k}{n}-\frac{1}{p}} (\diam D)^{k-\frac{n}{p}}.
$$
If $k>n$, then we have
$$
\left( \int_{D\setminus A} \frac{dz}{|x-z|^{(n-k)\frac{p}{p-1}}}\right)^{\frac{p-1}{p}} \leq
(\diam D)^{k-n} |D\setminus A|^{\frac{p-1}{p}} \leq 
C(n,p) \delta^{1-\frac{1}{p}} (\diam D)^{k-\frac{n}{p}}.
$$
Hence
$$
\diam f(D)=\sup_{x,y\in D} |f(x)-f(y)|\leq 2\sup_{x\in D}|f(x)-f_A| \leq
C(n,k,p)\eps(\diam D)^{k-\frac{n}{p}}\Vert D^k f\Vert_{L^p(D)}.
$$

The proof is complete.
\end{proof}

We will also need the following classical Besicovitch covering lemma, see e.g. \cite[Theorem~1.3.5]{ziemer}
\begin{lemma}[Besicovitch]
\label{L7}
Let $E\subset\bbbr^n$ and let $\{B_x\}_{x\in E}$ be a family of closed balls $B_x=\overline{B}(x,r_x)$ so that
$\sup_{x\in E}\{ r_x\}<\infty$. Then there is a countable (possibly finite) subfamily $\{ B_{x_i}\}_{i=1}^\infty$
with the property that
$$
E\subset\bigcup_{i=1}^\infty B_{x_i}
$$
and no point of $\bbbr^n$ belongs to more than $C(n)$ balls.
\end{lemma}

\section{Proof of Theorem \ref{mainTheorem}}
\label{sec:mainproof}

As we pointed out in Introduction we may assume that $m\leq n$ and $k\geq 2$. It is also easy to see that we can assume that $\Omega=\bbbr^n$ and
$f\in W^{k,p}(\bbbr^n,\bbbr^m)$. Indeed, it suffices to prove the claim of Theorem~\ref{mainTheorem} on compact subsets of $\Omega$ and so we may multiply $f$
by a compactly supported smooth cut-off function to get a function in $W^{k,p}(\bbbr^n,\bbbr^m)$.

We will prove the result using induction with respect to $n$.

If $n=1$, then $m=n=1$.
This gives $n-m-k+1 = 1-k \leq 0$ for any $k \in \bbbn$, so $\ell = 0$.
Thus the theorem is a direct consequence of Theorem~\ref{coarea}.

We shall prove now the theorem for $n\geq 2$ assuming that it is true in dimensions less than or equal to $n-1$.

Fix $p$ and integers $m$ and $k$ satisfying $n<p<\infty$, $m\leq n$, and $k \geq 2$. Write $\ell = \max(n-m-k+1,0)$.
Let $f \in W^{k,p}(\bbbr^n,\bbbr^m)$.

We can write
$$
C_f = K \cup A_1 \cup \dots \cup A_{k-1},
$$
where
$$
K := \{x \in C_f \, | \, 0< \text{rank} \, Df(x) < m \}
$$
and
$$
A_{s} := \{ x \in \bbbr^n \, | \, D^\alpha f(x)=0 \; \text{for all} \; 1 \leq |\alpha| \leq s \}
$$
Note that $A_1\supset A_2\supset\ldots\supset A_{k-1}$
is a decreasing sequence of sets.

In the first step, we will show that $A_{k-1} \cap f^{-1}(y)$ is $\ell$-null for a.e. $y \in \bbbr^m$.
Then we will prove the same for $(A_{s-1}\setminus A_s)\cap f^{-1}(y)$ for $s=2,3,\ldots,k-1$. To do this
we will use the Implicit Function and Kneser-Glaeser theorems to reduce our problem to a lower dimensional one 
and apply the induction hypothesis. Finally, we will consider the set $K$ and use a change of variables to show 
that we can reduce the dimension in the domain and in the target 
so that the fact that $\H^\ell(K\cap f^{-1}(y))=0$ will follow
from the induction hypothesis.

\begin{claim} 
\label{claim1}
$\mathcal{H}^{\ell}(A_{k-1} \cap f^{-1}(y))=0$ for a.e. $y \in \mathbb{R}^m$.
\end{claim}
\begin{proof}
Suppose $x \in A_{k-1}$.
Notice that $T_x^{k-1} f(y) = f(x)$ for any $y \in \bbbr^n$ since $D^{\alpha} f(x) = 0$
for every $1 \leq |\alpha| \leq k-1$. By Lemma~\ref{L3} applied
to each coordinate of $f = (f_1, \dots , f_m)$, we have for any cube 
$Q \subset \bbbr^n$ containing $x$ and any $y \in Q$,
\begin{equation} 
\label{morrey}
|f(y)-f(x)| \leq C (\diam Q)^{k-\frac{n}{p}} \left( \int_{Q} |D^kf (z)|^p \, dz \right)^{1/p}.
\end{equation}
Hence
\begin{equation}
\label{morrey1}
\diam f(Q) \leq C (\diam Q)^{k-\frac{n}{p}} \left( \int_{Q} |D^kf (z)|^p \, dz \right)^{1/p}.
\end{equation}

Let 
$F_1 := \{ x \in A_{k-1} \, | \, x \text{ is a density point of } A_{k-1} \}$ and
$F_2 := A_{k-1} \setminus F_1$. We will treat the sets $F_1 \cap f^{-1}(y)$ and $F_2 \cap f^{-1}(y)$
separately.

\noindent
{\bf Step 1.}
First we will prove that $\mathcal{H}^{\ell}(F_2 \cap f^{-1}(y))=0$ for almost every $y \in \mathbb{R}^m$.

Let $0<\eps<1$. Since $\H^n(F_2)=0$, there is an open set $F_2\subset U\subset\Omega$ such that
$\H^n(U)<\eps^{\frac{p}{p-m}}$.
For any $j\geq 1$ let $\{ Q_{ij}\}_{i=1}^\infty$ be a collection of closed cubes with pairwise disjoint interiors such that
$$
Q_{ij}\cap F_2\neq \emptyset,
\quad
F_2\subset\bigcup_{i=1}^\infty Q_{ij}\subset U,
\quad
\diam Q_{ij}<\frac{1}{j}.
$$
Since $F_2\cap Q_{ij}\neq\emptyset$, \eqref{morrey1} yields
$$
\H^m(f(Q_{ij})) \leq C (\diam f(Q_{ij}))^m \leq
C(\diam Q_{ij})^{m(k-\frac{n}{p})}\left(\int_{Q_{ij}} |D^k f(x)|^p \, dx\right)^{m/p}.
$$

\noindent
{\bf Case:} $n-m-k+1\leq 0$ so $\ell=0$.

This condition easily implies that $mk\geq n$ so we also have
$\frac{mp}{p-m}(k-\frac{n}{p}) \geq n$, and by H\"{o}lder's inequality,
\begin{eqnarray}
\H^m(f(F_2))
& \leq &
\sum_{i=1}^\infty \H^m(f(Q_{ij})) 
\leq
C\sum_{i=1}^\infty (\diam Q_{ij})^{m(k-\frac{n}{p})}\left(\int_{Q_{ij}} |D^k f(x)|^p\, dx\right)^{m/p} \nonumber \\
&\leq &
C \left( \sum_{i=1}^{\infty} (\diam Q_{ij})^{\frac{pm}{p-m}(k-\frac{n}{p})} \right)^{\frac{p-m}{p}} \left( \int_{\bigcup_{i=1}^{\infty} Q_{ij}} |D^k f(x)|^p \, dx\right)^{m/p} \nonumber \\
& \leq &
C\H^n(U)^{\frac{p-m}{p}}\left(\int_U |D^kf(x)|^p\, dx\right)^{m/p}<C\eps\Vert D^k f\Vert_p.
\label{square}
\end{eqnarray}
Since $\eps>0$ can be arbitrarily small, $\H^m(f(F_2))=0$ and hence $F_2\cap f^{-1}(y)=\emptyset$, i.e. $\H^\ell(F_2\cap f^{-1}(y))=0$ 
for a.e. $y\in\bbbr^m$.

\noindent
{\bf Case:} $\ell=n-m-k+1>0$.

The sets $\{ Q_{ij}\cap f^{-1}(y)\}_{i=1}^\infty$ form a covering of $F_2\cap f^{-1}(y)$ by sets of diameters less than $1/j$. Since
$$
\diam (Q_{ij}\cap f^{-1}(y)) \leq (\diam Q_{ij}) \chi_{f(Q_{ij})}(y)
$$
the definition of the Hausdorff measure yields

\begin{eqnarray}
\label{Natalia2}
\H^{\ell}(F_2\cap f^{-1}(y))
& \leq & 
C \liminf_{j\to\infty} \sum_{i=1}^\infty \diam (Q_{ij} \cap f^{-1}(y))^{\ell} \\ \nonumber
& \leq &
C \liminf_{j\to\infty} \sum_{i=1}^\infty (\diam Q_{ij})^{\ell} \chi_{f(Q_{ij})}(y).
\end{eqnarray}
We would like to integrate both sides with respect to $y\in\bbbr^m$. Note that the function on the right hand side is measurable since the 
sets $f(Q_{ij})$ are compact. However measurability of the function $y\mapsto \H^{\ell}(F_2\cap f^{-1}(y))$ is far from being
obvious. To deal with this problem we will use the {\em upper integral} which for a non-negative function $g:X\to [0,\infty]$
defined $\mu$-a.e. on a measure space $(X,\mu)$ is defined as follows:
$$
\int_X^* g\, d\mu =
\inf\left\{ \int_X \phi\, d\mu:\, \text{$0\leq g\leq\phi$ and $\phi$ is $\mu$-measurable.}\right\}\, .
$$
An important property of the upper integral is that if $\int_X^*g\, d\mu=0$, then $g=0$ $\mu$-a.e.
Indeed, there is a sequence $\phi_i \geq g \geq 0$ such that $\int_X\phi_i\, d\mu\to 0$.
That means $\phi_i\to 0$ in $L^1(\mu)$. Taking a subsequence we get $\phi_{i_j}\to 0$ $\mu$-a.e. which proves that $g=0$ $\mu$-a.e.

Applying the upper integral with respect to $y\in\bbbr^m$ to both sides of \eqref{Natalia2}, using Fatou's lemma, and noticing that
$$
\frac{p}{p-m} \left(\ell + m \left(k-\frac{n}{p} \right) \right) \geq n
$$
gives
\begin{eqnarray*}
\lefteqn{\int_{\bbbr^m}^* \H^{\ell} (F_2\cap f^{-1}(y))\, d\H^m(y)
\leq 
C\liminf_{j\to\infty} \sum_{i=1}^\infty (\diam Q_{ij})^{\ell} \H^m(f(Q_{ij}))}\\ 
& \leq &
C\liminf_{j\to\infty} \sum_{i=1}^\infty (\diam Q_{ij})^{\ell + m (k-\frac{n}{p})} \left(\int_{Q_{ij}} |D^k f(x)|^p \, dx\right)^{m/p} < C\eps\Vert D^k f\Vert_p
\end{eqnarray*}
by the same argument as in \eqref{square}. 
Again, since $\eps>0$ can be arbitrarily small, we conclude that 
$\H^\ell(F_2\cap f^{-1}(y))=0$ for a.e. $y\in\bbbr^m$.

\noindent
{\bf Step 2.} It remains to prove that $\mathcal{H}^{\ell}(F_1 \cap f^{-1}(y))=0$ for almost every $y \in \mathbb{R}^m$.

The proof is similar to that in Step~1 and the arguments which are almost the same will be presented in a more sketchy form now.
In Step~1 it was essential that the set $F_2$ had measure zero. We will compensate the lack of this property now by the estimates from
Remark~\ref{Rem1}. 

It suffices to prove that for any cube $\tilde{Q}$, $\H^\ell(\tilde{Q}\cap F_1\cap f^{-1}(y))=0$ for a.e. $y\in\bbbr^m$.
Assume that $\tilde{Q}$ is in the interior of a larger cube $\tilde{Q}\Subset Q$.

By Remark~\ref{Rem1}, for each $x\in \tilde{Q}\cap F_1$ and $j\in\bbbn$ there is $0<r_{j x}<1/j$ such that
$$
\diam f(B(x,r_{j x})) 
\leq
j^{-1} r_{j x}^{k-\frac{n}{p}} \left( \int_{B(x,r_{j x})} |D^k f(z)|^p\, dz\right)^{1/p}.
$$
We may further assume that $B(x,r_{jx})\subset Q$.

Denote $B_{jx}=\overline{B}(x,r_{jx})$.
According to the Besicovitch Lemma~\ref{L7}, there is a countable subcovering
$\{ B_{j x_i}\}_{i=1}^\infty$ of $\tilde{Q}\cap F_1$ so that no point of $\bbbr^n$ belongs to more than $C(n)$ balls $B_{j x_i}$. 

\noindent
{\bf Case:} $n-m-k+1\leq 0$ so $\ell=0$.

We have $\frac{pm}{p-m}(k-\frac{n}{p}) \geq n$ as before, so
\begin{eqnarray*}
\H^m(f(\tilde{Q}\cap F_1))
& \leq &
C \sum_{i=1}^{\infty} \H^m(f(B_{jx_i}))
\leq
C j^{-m}\sum_{i=1}^\infty r_{jx_i}^{m(k-\frac{n}{p})} \left(\int_{B_{jx_i}} |D^k f(z)|^p\, dz\right)^{m/p} \\
& \leq &
C j^{-m}\left(\sum_{i=1}^\infty r_{jx_i}^n\right)^{\frac{p-m}{p}}\left(\sum_{i=1}^\infty \int_{B_{jx_i}} |D^k f(z)|^p\, dz\right)^{m/p}\, .
\end{eqnarray*}
Since the balls are contained in $Q$ and no point belongs to more than $C(n)$ balls we conclude that
$$
\H^m(f(\tilde{Q}\cap F_1))\leq C j^{-m} \H^n(Q)^{\frac{p-m}{p}}\Vert D^k f\Vert_p^m.
$$
Since $j$ can be arbitrarily large, $\H^m(f(\tilde{Q}\cap F_1))=0$, i.e.  
$\H^\ell(\tilde{Q}\cap F_1\cap f^{-1}(y))=0$ for a.e. $y\in\bbbr^m$.

\noindent
{\bf Case:} $\ell=n-m-k+1>0$.

The sets
$\{ B_{j x_i}\cap f^{-1}(y)\}_{i=1}^\infty$ form a covering of $\tilde{Q}\cap F_1\cap f^{-1}(y)$ and
$$
\diam (B_{j x_i}\cap f^{-1}(y)) \leq
C r_{j x_i} \chi_{f(B_{j x_i})}(y).
$$
The definition of the Hausdorff measure yields
$$
\H^{\ell} (\tilde{Q}\cap F_1\cap f^{-1}(y)) \leq
C\liminf_{j\to\infty} \sum_{i=1}^\infty r_{j x_i}^{\ell} \chi_{f(B_{j x_i})}(y).
$$
Thus as above
\begin{eqnarray*}
\lefteqn{\int_{\bbbr^m}^* \H^{\ell} (\tilde{Q}\cap F_1\cap f^{-1}(y))\, d\H^m(y)} \\
& \leq &
C \liminf_{j\to\infty} \sum_{i=1}^\infty r_{j x_i}^{\ell} \H^m(f(B_{j x_i})) \\
& \leq &
C \liminf_{j\to\infty} j^{-m} \sum_{i=1}^\infty
r_{j x_i}^{\ell + m(k-\frac{n}{p})} \left( \int_{B_{j x_i}} |D^k f(z)|^p\, dz\right)^{m/p} \\
& \leq &
C \liminf_{j\to\infty} j^{-m} \H^n(Q)^{\frac{p-m}{p}}\Vert D^k f\Vert_p^m = 0
\end{eqnarray*}
since $\frac{p}{p-m} \left(\ell + m \left(k-\frac{n}{p} \right) \right) \geq n$.
Therefore $\H^\ell(\tilde{Q}\cap F_1\cap f^{-1}(y))=0$ for a.e. $y\in\bbbr^m$.
This completes the proof that
$\H^{\ell}(F_1\cap f^{-1}(y))=0$ for a.e. $y\in\bbbr^m$ and hence that of Claim~\ref{claim1}
\end{proof}

\begin{claim} 
\label{claim2}
$\mathcal{H}^{\ell}((A_{s-1} \setminus A_s) \cap f^{-1}(y))=0$ for a.e. $y \in \mathbb{R}^m$, $s=2,3, \dots , k-1$.
\end{claim}

In this step, we will use the Kneser-Glaeser composition theorem and the implicit function theorem to apply the induction hypothesis in $\mathbb{R}^{n-1}$.

Fix $s \in \{ 2,3, \dots , k-1 \}$ and $\bar{x} \in A_{s-1} \setminus A_s$. 
It suffices to show that the $\ell$-Hausdorff measure of $W \cap (A_{s-1} \setminus A_s) \cap f^{-1}(y)$ 
is zero for some neighborhood $W$ of $\bar{x}$ and a.e. $y \in \bbbr^m$. 
Indeed, $A_{s-1} \setminus A_s$ can be covered by countably many such neighborhoods.

By the definitions of $A_s$ and $A_{s-1}$, $D^{\gamma}f(\bar{x})=0$ for all $1 \leq |\gamma| \leq s-1$, 
and $D^{\beta}f(\bar{x}) \neq 0$ for some $|\beta| = s$. 
That is, for some $|\gamma| = s-1$ and $j \in \{1 ,\dots, m \}$, $D (D^{\gamma}f_j)(\bar{x}) \neq 0$ 
and $D^{\gamma}f_j \in W^{k-(s-1),p} \subset C^{k-s,1-\frac{n}{p}}$.

Hence, by the implicit function theorem, 
there is some neighborhood $U$ of $\bar{x}$ and an open set $V \subset \mathbb{R}^{n-1}$ 
so that $U \cap \{ D^{\gamma}f_j=0 \} = g(V)$ for some $g: V \to \mathbb{R}^n$ of class $C^{k-s}$. 
In particular, $U \cap A_{s-1} \subset g(V)$ since $D^{\gamma}f_j=0$ on $A_{s-1}$.

Choose a neighborhood $W \Subset U$ of $\bar{x}$ and say $A^* := g^{-1}(\overline{W} \cap A_{s-1})$ so that $A^*$ is compact.
Since $f$ is $s-1$ flat on the closed set $A_{s-1}$, $f$ is of class $C^{k-1}$, $g$ is of class $C^{(k-1)-(s-1)}$, and $g(A^*) \subset A_{s-1}$, 
we can apply Theorem \ref{rough} to each component of $f$ to find a $C^{k-1}$ function $F:\bbbr^{n-1} \to \mathbb{R}^m$ so that, 
for every $x \in A^*$, $F(x) = (f \circ g)(x)$ and $D^{\lambda}F(x)=0$ for all $|\lambda| \leq s-1$. That is, $A^* \subset C_F$. 
Hence
$$
\mathcal{H}^{\ell}(A^* \cap F^{-1}(y)) 
\leq \mathcal{H}^{\ell}(C_F \cap F^{-1}(y)) = 0.
$$
for almost every $y \in \mathbb{R}^m$.
In this last equality, we invoked the induction hypothesis on $F \in C^{k-1}(\bbbr^{n-1},\bbbr^m) \subset W_{\rm loc}^{k-1,p}(\bbbr^{n-1},\bbbr^m)$ with $\ell = \max((n-1)-m-(k-1)+1,0)$.
Since $g$ is of class $C^1$, it is locally Lipschitz, and so 
$\mathcal{H}^{\ell}(g(A^* \cap F^{-1}(y)))=0$
for almost every $y \in \mathbb{R}^m$. Since $W \cap A_{s-1} \subset g(A^*)$, we have
$$
W \cap A_{s-1} \cap f^{-1}(y) \subset g(A^* \cap F^{-1}(y))
$$
for all $y \in \mathbb{R}^m$, and thus
$$\mathcal{H}^{\ell}(W \cap (A_{s-1} \setminus A_s) \cap f^{-1}(y)) \leq \mathcal{H}^{\ell}(W \cap A_{s-1} \cap f^{-1}(y)) = 0$$
for almost every $y \in \mathbb{R}^m$. The proof of the claim is complete.

\begin{claim}
\label{Michal}
$\mathcal{H}^{\ell}(K \cap f^{-1}(y))=0$ for a.e. $y \in \mathbb{R}^m$.
\end{claim}
\begin{proof}
Write $K = \bigcup_{r=1}^{m-1} K_r$ where $K_r := \{ x \in \bbbr^n \, | \, \text{rank} \, Df(x)=r \}$. 
Fix $x_0 \in K_r$ for some $r \in \{ 1, \dots , m-1 \}$. 
For the same reason as in Claim~\ref{claim2} it suffices to
show that $\mathcal{H}^{\ell}((V \cap K_r) \cap f^{-1}(y))=0$ for some neighborhood $V$ of $x_0$ for a.e. $y \in \mathbb{R}^m$.

Without loss of generality, assume that the submatrix $\left[ \partial f_i/\partial x_j (x_0) \right]_{i,j=1}^r$ 
formed by the first $r$ rows and columns of $Df$ has rank $r$. Let
\begin{equation}
\label{M1}
Y(x) = (f_1(x), f_2(x), \dots , f_r(x), x_{r+1}, \dots , x_n) \quad 
\text{for all $x\in\bbbr^n$.}
\end{equation}
$Y$ is of class $C^{k-1}$ since each component of $f$ is. Also, $\text{rank} \, DY(x_0) = n$, so by the inverse function theorem $Y$ is a 
$C^{k-1}$ diffeomorphism of some neighborhood $V$ of $x_0$ onto an open set $\tilde{V} \subset \mathbb{R}^n$.
From now on we will assume that $Y$ is defined in $V$ only.

\begin{claim}
\label{MC1}
$Y^{-1}\in W^{k,p}_{\rm loc}(\tilde{V},\bbbr^n)$.
\end{claim}
\begin{proof}
In the proof we will need
\begin{lemma}
\label{ML1}
Let $\Omega\subset\bbbr^n$ be open. If $g,h\in W^{\ell,p}_{\rm loc}(\Omega)$, where $p>n$ and $\ell\geq 1$, then $gh\in W^{\ell,p}_{\rm loc}(\Omega)$.
\end{lemma}
\begin{proof}
Since $g,h\in C^{\ell-1}$, it suffices to show that the classical partial derivatives $D^{\beta}(gh)$, $|\beta|=\ell-1$ belong to $W^{1,p}_{\rm loc}(\Omega)$
(when $\ell=1$, $\beta=0$ so $D^\beta(gh)=gh$). 

The product rule for $C^{\ell-1}$ functions yields
\begin{equation}
\label{M2}
D^\beta(gh)=\sum_{\gamma+\delta=\beta} \frac{\beta!}{\gamma!\,\delta!} \, D^\gamma g D^\delta h.
\end{equation}
Each of the functions $D^\gamma g$, $D^\delta h$ is absolutely continuous on almost all lines parallel to coordinate axes, \cite[Section~4.9.2]{Evans},
so is their product. Thus $D^\beta(gh)$ is absolutely continuous on almost all lines and hence it has partial derivatives (or order $1$)
almost everywhere. According to a characterization of $W^{1,p}_{\rm loc}$ by absolute continuity on lines, \cite[Section~4.9.2]{Evans},
it suffices to show that partial derivatives of $D^\beta(gh)$ (of order $1$) belong to $L^p_{\rm loc}$. This will imply that
$D^\beta(gh)\in W^{1,p}_{\rm loc}$ for all $\beta$, $|\beta|=\ell-1$ so $gh\in W^{\ell,p}_{\rm loc}$.

If $D^\alpha =D^{\delta_i}D^{\beta}$,  then the product rule applied to the right hand side of \eqref{M2} yields
$$
D^\alpha(gh)=\sum_{\gamma+\delta=\alpha} \frac{\alpha!}{\gamma!\,\delta!}\, D^\gamma g D^\delta h.
$$
If $|\gamma|<|\alpha|=\ell$ and $|\delta|<|\alpha|=\ell$, then the function $D^\gamma g D^\delta h$ is continuous and hence in $L^p_{\rm loc}$.
The remaining terms are $hD^\alpha g+g D^\alpha h$. Clearly this function also belongs to $L^p_{\rm loc}$ because the functions $g,h$ are continuous and 
$D^\alpha g, D^\alpha h\in L^p_{\rm loc}$.
This completes the proof of the lemma.
\end{proof}
Now we can complete the proof of Claim~\ref{MC1}. Since $Y$ is a diffeomorphism of class $C^{k-1}$, we have
\begin{equation}
\label{M3}
D(Y^{-1})(y)=[DY(Y^{-1}(y))]^{-1}
\quad
\text{for every $y\in\tilde{V}$}.
\end{equation}
It suffices to prove that $D(Y^{-1})\in W^{k-1,p}_{\rm loc}$. It follows from \eqref{M3} and a formula for the inverse matrix that
$$
D(Y^{-1}) = \left(\frac{P_1(Df)}{P_2(Df)}\right)\circ Y^{-1},
$$
where $P_1$ and $P_2$ and polynomials whose variables are replaced by partial derivatives of $f$. The polynomial $P_2(Df)$ is just
$\det DY$.

Since $Df\in W^{k-1,p}_{\rm loc}$ and $p>n$, it follows from Lemma~\ref{ML1} that
$$
P_1(Df), P_2(Df)\in W^{k-1,p}_{\rm loc}.
$$
Note that $P_2(Df)=\det DY$ is continuous and different than zero. Hence
$$
\frac{1}{P_2(Df)}\in W^{k-1,p}_{\rm loc}
$$
as a composition of a $W^{k-1,p}_{\rm loc}$ function which is locally bounded away from $0$ and $\infty$
with a smooth function $x\mapsto x^{-1}$. Thus Lemma~\ref{ML1} applied one more time yields that
$P_1(Df)/P_2(Df)\in W^{k-1,p}_{\rm loc}$. Finally
$$
D(Y^{-1}) = \left(\frac{P_1(Df)}{P_2(Df)}\right)\circ Y^{-1}\in W^{k-1,p}_{\rm loc}
$$
because composition with a diffeomorphism $Y^{-1}$ of class $C^{k-1}$ preserves $W^{k-1,p}_{\rm loc}$.
The proof of the claim is complete.
\end{proof}

It follows directly from \eqref{M1} that
\begin{equation}
\label{M4}
f(Y^{-1}(x))=(x_1,\ldots,x_r,g(x))
\end{equation}
for all $x\in\tilde{V}$ and some function $g:\tilde{V}\to\bbbr^{m-r}$.
\begin{claim}
\label{MC2}
$g\in W^{k,p}_{\rm loc}(\tilde{V},\bbbr^{m-r})$.
\end{claim}
This statement is a direct consequence of the next
\begin{lemma}
\label{ML2}
Let $\Omega\subset\bbbr^n$ be open, $p>n$ and $k\geq 1$. If $\Phi\in W^{k,p}_{\rm loc}(\Omega,\bbbr^n)$ is a diffeomorphism and 
$u\in W^{k,p}_{\rm loc}(\Phi(\Omega))$, then $u\circ\Phi\in W^{k,p}_{\rm loc}(\Omega)$.
\end{lemma}
\begin{proof}
When $k=1$ the result is obvious because diffeomorphisms preserve $W^{1,p}_{\rm loc}$. Assume thus that $k\geq 2$.
Since $p>n$, $\Phi\in C^{k-1}$ so $\Phi$ is a diffeomorphism of class $C^{k-1}$, but also $u\in C^{k-1}\subset C^1$ and hence the
classical chain rule gives
\begin{equation}
\label{M5}
D(u\circ\Phi)=((Du)\circ\Phi)\cdot D\Phi.
\end{equation}
Since $Du\in W^{k-1,p}_{\rm loc}$ and $\Phi$ is a diffeomorphism of class $C^{k-1}$,
we conclude that $(Du)\circ\Phi\in W^{k-1,p}_{\rm loc}$. Now the fact that $D\Phi\in W^{k-1,p}_{\rm loc}$ combined with \eqref{M5} and Lemma~\ref{ML1} yield that
the right hand side of \eqref{M5} belongs to $W^{k-1,p}_{\rm loc}$ so $D(u\circ\Phi)\in W^{k-1,p}_{\rm loc}$ and hence
$u\circ\Phi\in W^{k,p}_{\rm loc}$. This compltes the proof of Lemma~\ref{ML2} and hence that of Claim~\ref{MC2}.
\end{proof}

Now we can complete the proof of Claim~\ref{Michal}.
Recall that we need to prove that
\begin{equation}
\label{M6}
\H^\ell((V\cap K_r)\cap f^{-1}(y))=0
\quad
\text{for a.e. $y\in \bbbr^m$.}
\end{equation}
The diffeomorphism $Y^{-1}$ is a change of variables that simplifies the structure of the mapping $f$ because $f\circ Y^{-1}$ fixes the first $r$ coordinates
(see \eqref{M4}) and hence it maps $(n-r)$-dimensional slices orthogonal to $\bbbr^r$ to the corresponding $(m-r)$-dimensional slices orthogonal to $\bbbr^r$.
Because of this observation it is more convenient to work with $f\circ Y^{-1}$ rather than with $f$. Translating \eqref{M6} to the case of $f\circ Y^{-1}$
it suffices to show that
$$
\H^\ell((\tilde{V}\cap Y(K_r))\cap (f\circ Y^{-1})^{-1})(y)=0
\quad
\text{for a.e. $y\in\bbbr^m$}.
$$
We used here a simple fact that the diffeomorphism $Y$ preserves $\ell$-null sets. 

Observe also that
\begin{equation}
\label{M7}
\rank D(f\circ Y^{-1})(x)=r
\quad
\text{for $x\in\tilde{V}\cap Y(K_r)$}.
\end{equation}

For any $\tilde{x} \in \bbbr^r$ and $A \subset \bbbr^n$, we will denote by $A_{\tilde{x}}$ the $(n-r)$--dimensional slice of $A$ with the first 
$r$ coordinates equal to $\tilde{x}$. That is, $A_{\tilde{x}} := \{ z \in \mathbb{R}^{n-r} \, | \, (\tilde{x} , z) \in A \}$.
Let $g_{\tilde{x}}:\tilde{V}_{\tilde{x}}\to\bbbr^{m-r}$ be defined by $g_{\tilde{x}}(z)=g(\tilde{x},z)$. With this notation
$$
(f\circ Y^{-1})(\tilde{x},z)=(\tilde{x},g_{\tilde{x}}(z))
$$
and hence for $y=(\tilde{x},w)\in \bbbr^m$
$$
(\tilde{V}\cap Y(K_r))\cap (f\circ Y^{-1})^{-1}(y)=
g_{\tilde{x}}^{-1}(w)\cap (\tilde{V}\cap Y(K_r))_{\tilde{x}}.
$$
More precisely the set on the left hand side is contained in an affine 
$(n-r)$-dimensional subspace of $\bbbr^n$ orthogonal to $\bbbr^r$ at $\tilde{x}$ while the set on the right hand side is
contained in $\bbbr^{n-r}$ but the two sets are identified through a translation by the vector $(\tilde{x},0)\in\bbbr^n$ which 
identifies $\bbbr^{n-r}$ with the affine subspace orthogonal to $\bbbr^r$ at $\tilde{x}$.

According to the Fubini theorem it suffices to show that for almost all $\tilde{x}\in\bbbr^r$ the following is true:
for almost all $w\in\bbbr^{m-r}$
\begin{equation}
\label{M8}
\H^\ell (g_{\tilde{x}}^{-1}(w)\cap (\tilde{V}\cap Y(K_r)_{\tilde{x}}))=0.
\end{equation}
As we will see this is a direct consequence of the induction hypothesis applied to the mapping
$g_{\tilde{V}}:\tilde{x}_{\tilde{x}}\to\bbbr^{n-r}$ defined in a set of dimension $n-r\leq n-1$. 
We only need to check that $g_{\tilde{x}}$ satisfies the assumptions of the induction hypothesis.

It is easy to see that for each $x=(\tilde{x},z)\in\tilde{V}$
$$
D(f \circ Y^{-1})(x) = 
\left( \begin{array}{ccc}
\id_{r \times r} & 0 \\
* & D(g_{\tilde{x}})(z) \end{array} \right).
$$
This and \eqref{M7} imply that for each $\tilde{x}\in\bbbr^r$, $Dg_{\tilde{x}}=0$ on the slice $(\tilde{V}\cap Y(K_r))_{\tilde{x}}$. Hence
the set $(\tilde{V}\cap Y(K_r))_{\tilde{x}}$ is contained in the critical set of $g_{\tilde{x}}$ so
\begin{equation}
\label{M9}
\H^\ell(g_{\tilde{x}}^{-1}(w)\cap(\tilde{V}\cap Y(K_r))_{\tilde{x}})\leq 
\H^\ell(g_{\tilde{x}}^{-1}(w)\cap C_{g_{\tilde{x}}}).
\end{equation}
It follows from the Fubini theorem applied to Sobolev spaces that for almost all $\tilde{x}\in\bbbr^n$,
$g_{\tilde{x}}\in W^{k,p}_{\rm loc}(\tilde{V}_{\tilde{x}},\bbbr^{m-r})$ and hence the induction hypothesis is satisfied for such mappings
$$
W^{k,p}_{\rm loc}\ni g_{\tilde{x}}:\tilde{V}_{\tilde{x}}\subset\bbbr^{n-r}\to\bbbr^{m-r}.
$$
Since
$$
\ell=\max(n-m-k+1,0)=\max((n-r)-(m-r)-k+1,0),
$$
for almost all $w\in \bbbr^{m-n}$ the expression on the right hand side of \eqref{M9} equals zero and \eqref{M8} follows.
This completes the proof of Claim~\ref{Michal} and hence that of the theorem.
\end{proof}

\end{document}